\documentclass[11pt]{article}
\usepackage{amsthm}
\usepackage{amssymb}
\usepackage{amsmath,enumerate}
\usepackage{comment}
\usepackage{thm-restate}
\usepackage{url} 
\usepackage{hyperref}
\usepackage{graphicx}
\usepackage{subfigure}
\usepackage[noabbrev,capitalise]{cleveref}
\usepackage[affil-it]{authblk}
\usepackage{color}
\usepackage[normalem]{ulem}

\usepackage{tikz}
\usetikzlibrary{positioning}
\usepackage[margin=1in]{geometry}
\usepackage{tikz}

\theoremstyle{plain}
\newtheorem{thm}{Theorem}[section]
\newtheorem{lem}[thm]{Lemma}

\newtheorem{cor}[thm]{Corollary}

\newtheorem{conj}[thm]{Conjecture}

\theoremstyle{plain} 
\newcommand{\thistheoremname}{}
\newtheorem{genericthm}[section]{\thistheoremname}

\theoremstyle{definition}

\def\es{\emptyset}
\def\less{\setminus}

\newcounter{counter}
 
\def\dfn#1{{\sl #1}}

\def\es{\emptyset}
\def\less{\setminus}

\newcommand{\ceil}[1]{{\left\lceil #1 \right\rceil}}

\title{Dominating  Hadwiger's Conjecture for graphs $G$ with $\alpha(G)=2$}
\author{Michael Scully\thanks{Supported in part by  NSF grant DMS-2153945.  E-mail address: {\tt  Michael.Scully@ucf.edu}.}\,\, and  Zi-Xia Song\thanks{Supported by NSF grant DMS-2153945.  E-mail address: {\tt Zixia.Song@ucf.edu}.}}
 
  \affil{ 
  { \small {Department  of Mathematics, University of Central Florida, Orlando, FL 32816, USA}}  
     }

 \date{}

\begin{document}
\maketitle
\begin{abstract}
      Hadwiger's Conjecture from 1943 states that every graph with chromatic number $t$  contains a  $K_t$ minor.   Illingworth and Wood [arXiv:2405.14299]  introduced the concept  of  a ``dominating $K_t$ minor'' and asked whether   every graph with chromatic number $t$   contains a dominating $K_t$ minor.  This question is   a  substantial strengthening of   Hadwiger’s Conjecture.    Norin  referred to it as the ``Dominating Hadwiger's Conjecture'' and believes it is likely false.  In this paper   we first observe that a $t$-chromatic $G$ on $n$ vertices with independence number $\alpha(G)\le2$   contains a dominating  $K_t$ minor if and only if $G$ contains a dominating   $K_{\lceil{n/2}\rceil}$ minor. 
Building on this and using a deep result of Chudnovsky and Seymour on  packing seagulls,  we prove that every graph $G$ on $n$ vertices with   $\alpha(G)\le 2$ and  $2\omega(G)\ge \lceil{n/2}\rceil+1$ satisfies   the Dominating Hadwiger's Conjecture, where $\omega(G)$ denotes the clique number of $G$. We  further prove that  every $H$-free graph  $G$ with $\alpha(G)\le 2$ satisfies  the Dominating Hadwiger's Conjecture, where    $H\in\{2K_1+P_4,   K_2+2K_2,   K_2+(K_1\cup K_3), K_1+(K_1\cup K_5), W_5^<,  W_5^-, W_5,   K_7^<, K_7^-, K_7\}$ or  $H\ne K_2\cup K_3$ is a  graph on at most five vertices  such that  $\alpha(H)\le2$.

\end{abstract}
  
\baselineskip 16pt

\section{Introduction}

All graphs in this paper are finite and simple. For a graph $G$, we use $|G|$, $\alpha(G)$, $\chi(G)$, $\omega(G)$, $\delta(G)$ and $\Delta(G)$ to denote the number of vertices,    independence number, chromatic number,  clique number, minimum degree and maximum degree of $G$, respectively. The \dfn{complement} of $G$ is denoted by $\overline{G}$. Given a graph $H$, we say that $G$ is \dfn{$H$-free} if $G$ has no induced subgraph isomorphic to $H$.   The join $G + H$ (resp. union $G \cup H$) of two vertex-disjoint graphs $G$ and $H$ is the graph having vertex set $V(G) \cup V(H)$ and edge set $E(G) \cup E(H) \cup \{xy \mid x\in V(G), y \in V(H) \}$ (resp. $E(G) \cup E(H)$).  For a positive integer $k$, we use $kG$ to denote $k$ disjoint copies of the graph $G$. 
\medskip

   A graph  $G$ contains a  \dfn{$K_t$ minor} if  there exist pairwise disjoint non-empty connected subgraphs $T_1, \ldots, T_t$ of $G$, such that for $1 \leq i<j\leq t$, some vertex in $T_j$ has a neighbor in $T_i$.  Our work is motivated by the  celebrated Hadwiger's Conjecture~\cite{Had43}.

\begin{conj}[Hadwiger's Conjecture~\cite{Had43}]\label{c:HC} For every integer $t \geq 1$, every graph with no $K_t$ minor is $(t-1)$-colorable. 
\end{conj}

Conjecture~\ref{c:HC} is  trivially true for $t\le3$, and reasonably easy for $t=4$, as shown independently by Hadwiger~\cite{Had43} and Dirac~\cite{Dirac52}. However, for $t\ge5$, Hadwiger's Conjecture implies the Four Color Theorem~\cite{AH77,AHK77}.   Wagner~\cite{Wagner37} proved that the case $t=5$ of Hadwiger's Conjecture is, in fact, equivalent to the Four Color Theorem, and the same was shown for $t=6$ by Robertson, Seymour and  Thomas~\cite{RST93}. Despite receiving considerable attention over the years, Hadwiger's Conjecture remains  open for $t\ge 7$ and is   widely considered among the most important problems in graph theory and has motivated numerous developments in graph coloring and graph minor theory. In the 1980s, Kostochka~\cite{Kostochka82,Kostochka84}  and Thomason~\cite{Thomason84} proved that the  upper bound on the chromatic number of graphs with no $K_t$ minor  is   $ O(t \sqrt{\log t})$,  which was improved to $O(t (\log t)^{1/4 + \epsilon})$ by Norin, Postle, and the first author~\cite{NPS23}, and very recently, to $ O(t\log\log t)$ by Delcourt and Postle \cite{DP25}.
.\medskip

  Recently, Illingworth and Wood~\cite{IW24}  introduced the notion of a  ``dominating $K_t$  minor''. A graph $G$ contains a \dfn{dominating $K_t$ minor} if there exists a  sequence $(T_1,\ldots,T_t)$ of pairwise disjoint non-empty connected subgraphs of $G$, such that for $1 \leq i<j\leq t$, every vertex in $T_j$ has a neighbor in $T_i$. Replacing ``every vertex in $T_j$'' by ``some vertex in $T_j$'' retrieves the standard definition of a $K_t$ minor  of $G$. 
 Illingworth and Wood~\cite{IW24}  explored in what sense dominating $K_t$ minors  behave like (non-dominating) $K_t$ minors: they observed   that the two notions are equivalent for $t\leq 3$, but are already very different for $t=4$, for instance, the $1$-subdivision of $K_n$ contains a $K_n$ minor for all $n \geq 4$, yet lacks a dominating $K_4$ minor. This led them to propose a natural strengthening of Hadwiger’s Conjecture:  \emph{is every graph  with no dominating $K_t$ minor  $(t-1)$-colorable?}  Norin~\cite{Norin24} referred to this as the ``Dominating Hadwiger's Conjecture'' and    suspects the conjecture is false.   
 
\begin{conj}[Dominating Hadwiger's Conjecture~\cite{IW24}]\label{c:dHC}
    For every integer $t \geq 1$, every graph with no dominating $K_t$ minor is $(t-1)$-colorable.
\end{conj}
 
Conjecture~\ref{c:dHC} is  trivially true for $t\le3$.  Illingworth and Wood~\cite{IW24} verified the case $t=4$ by  showing  that every graph with no dominating $K_4$ minor  is $2$-degenerate and $3$-colorable. Norin~\cite{Norin24} announced a proof that every graph with no dominating $K_5$ minor  is $5$-degenerate.  More generally, Illingworth and Wood~\cite{IW24}  proved that 
every graph with no dominating $K_t$ minor  is $2^{t-2}$-colorable. 
  They further  provided two pieces of evidence for the Dominating Hadwiger's  
Conjecture: (1) It is true for almost every graph, (2) Every graph $G$ with no dominating $K_t$ minor has a $(t-1)$-colorable induced subgraph on at least half the vertices, which implies there is an independent set of size at least $\ceil{|G|/(2t-2)}$.  \medskip

The \dfn{dominating Hadwiger number} $h_d(G)$ of a graph $G$ is the largest integer $t$ such that $G$ contains a dominating $K_t$ minor.  Very recently, in their attempt  to find potential counterexamples, Tibbetts and the second author~\cite{SongTibbetts25} proved that  the Dominating Hadwiger's Conjecture holds for all $2K_2$-free graphs. 

\begin{thm}[Song and Tibbetts~\cite{SongTibbetts25}]\label{t:2K2} 
 Every $2K_2$-free graph $G$  satisfies $h_d(G) \geq \chi(G)$.   
 \end{thm}

In the same paper Illingworth and Wood~\cite{IW24} also asked the following natural question: \emph{Does every $n$-vertex graph with $\alpha(G)=2$ have a dominating $K_t$ minor  with $t\geq \ceil{\frac{n}{2}}$?}  They further observed  that if $\alpha(G)=2$,  then $G$ has a dominating $K_t$ minor, where $t \geq \ceil{\frac{n}{3}}$. 
 It is known that    $ \chi(G) \cdot  \alpha(G)\ge  |G|  $ for every graph $G$. It is then natural to ask whether the following  weaker  conjecture holds. 

\begin{conj}\label{c:dDM}
    Every graph $G$ on $n$ vertices satisfies $h_d(G) \ge \ceil{n/\alpha(G)} $.
\end{conj}

   We first observe that Conjecture~\ref{c:dHC} and Conjecture~\ref{c:dDM} are equivalent for graphs with independence number at most two.  Following the ideas of Plummer, Stiebitz and Toft in~\cite{PST03}, we provide a proof here for completeness. 
 
\begin{thm}\label{t:equiv}
    Let $G$ be a graph on $n$ vertices with $\alpha(G) \leq 2$. Then  
   \[h_d(G) \geq \chi(G)  \text{ if and only if }   h_d(G) \geq \ceil{n/2}.\]
    
\end{thm}
\begin{proof}
      Since $\alpha(G)\le2$, we see that $\chi(G)\ge \ceil{n/2}$.  Suppose  $h_d(G) \geq \chi(G)$. Then  $h_d(G) \geq \chi(G)\ge\ceil{n/2}$. Conversely,  suppose $h_d(G) \geq \ceil{n/2}$ but $h_d(G) < \chi(G)$. We choose $G$ with $n$ minimum. Since $\chi(G)\ge \ceil{n/2}$,  we have $n\le 2\chi(G)$. If there exists a vertex $x\in V(G)$ such that  $\chi(G \less x) = \chi(G)$,   then  by the minimality of $n$,    $h_d(G)\ge h_d(G \less x)\ge \chi(G\less x)=\chi(G)$, a contradiction.   This proves that   $G$ is $\chi(G)$-critical.  Suppose  $n = 2\chi(G)$. Let $x\in V(G)$. By the minimality of $n$,     
\[h_d(G)\ge h_d(G\less x)\ge \chi(G \less x) \ge \ceil{ (n-1)/2} =\chi(G),\]
 a contradiction. Thus  $n \leq 2\chi(G) - 1$.  

Suppose next  $\overline{G}$ is disconnected.  Then $G$  is the join of  two  disjoint graphs, say  $G_1$ and $G_2$. By the minimality of $n$, $h_d(G_i) \ge  \chi(G_i)$ for each $i\in[2]$. It follows that  \[h_d(G)\ge h_d(G_1)+h_d(G_2)\ge \chi(G_1)+\chi(G_2)=\chi(G),\] 
a contradiction.  Thus $\overline{G}$ is connected.    By a deep result of Gallai~\cite{Gallai63} on the order of $\chi(G)$-critical graphs, we have $n = 2\chi(G) - 1$.    But then $h_d(G)<\chi(G)=  \ceil{n/2}$, contrary to the assumption that $h_d(G)\ge  \ceil{n/2}$. 
\end{proof}

We now focus on the Dominating Hadwiger's Conjecture  for graphs $G$ with $\alpha(G) \le 2$.   The complement of   such graphs are triangle-free.  As Plummer, Stiebitz and Toft point out in \cite{PST03}, this is a mild restriction considering the wide variety of triangle-free graphs.   In his survey, Seymour~\cite{Seymoursurvey}  remarks:
\begin{quote}
``This seems to me to be an excellent place to look for a counterexample. My own belief is, if it is true for graphs with stability number two then it is probably true in general, so it would be very nice to decide this case."
\end{quote}
Graphs with $\alpha(G)\le 2$ thus offers a promising setting to explore the Dominating Hadwiger's Conjecture.   Combining Theorem~\ref{t:equiv} with a deep result of   Chudnovsky and Seymour~\cite{seagull} (see Theorem~\ref{t:seagulls} in Section~\ref{s:omega}) on  packing seagulls, we prove that the Dominating Hadwiger's Conjecture   holds for graphs $G$ with $\alpha(G)\le 2$ and $2\omega(G)\ge \ceil{|G|/2}+1$. This extends the deep result of  Chudnovsky and Seymour~\cite[Theorem 1.3]{seagull} for Hadwiger's Conjecture. \medskip

\begin{restatable}{thm}{Clique}\label{t:omega}
 Let $G$ be a graph on $n$ vertices with $\alpha(G)\le 2$.  If $2\omega(G)\ge \ceil{n/2}+1$, then  $h_d(G) \geq \lceil n/2\rceil$, and so $G$ satisfies the Dominating Hadwiger's Conjecture. 
 \end{restatable}
  
Theorem~\ref{t:omega} combined with small Ramsey numbers immediately implies the following:

\begin{cor}\label{c:smallclique} Let $G$ be a graph  on $n$ vertices with $\alpha(G)\le2$. If   $\omega(G)\le 6$, then $h_d(G)\ge \ceil{n/2}$.
\end{cor}
\begin{proof} Since  the Ramsey number $R(3, \omega(G)+1))\le 4\omega(G)-1$ when $\omega(G)\le 6$, we see that $n\le 4\omega(G)-2$. Then $2\omega(G)\ge \ceil{n/2}+1$. By Theorem~\ref{t:omega},  $h_d(G)\ge \ceil{n/2}$.
\end{proof}

\begin{cor}\label{c:smalln} Let $G$ be a graph  on $n$ vertices with $\alpha(G)\le2$. If   $n\le 26$, then $h_d(G)\ge \ceil{n/2}$.
\end{cor}
\begin{proof} Suppose $h_d(G)< \ceil{n/2}$. By Theorem~\ref{t:omega},   $2\omega(G)\le \ceil{n/2} $.  By Corollary~\ref{c:smallclique},   $\omega(G)\ge 7$.    But then $14\le2\omega(G)\le\ceil{n/2} $ and so $n\ge27$,  contrary to the assumption that $n\le 26$.  
\end{proof}

  Let $K_k^-$ denote the graph obtained from the complete graph $K_k$ by deleting a single edge, and let $K_k^{<}$ denote the graph obtained by deleting two adjacent edges from $K_k$. The \dfn{wheel} $W_k$ on $k+1$ vertices is defined as the graph obtained from the cycle $C_k$ of length $k$ by adding a new vertex $x$ adjacent to every vertex of $C_k$. We denote by $W_k^-$ the graph obtained from $W_k$ by deleting one edge incident to $x$, and by $W_k^{<}$ the graph obtained by deleting two consecutive edges incident to $x$. Finally, we use $P_k$  to denote the path  on  $k$ vertices.   In our search for potential counterexamples,  we   further prove the following main result. 

\begin{restatable}{thm}{GraphH}\label{t:main}
 Let $G$ be a graph on $n$ vertices with $\alpha(G)\le 2$. Let $H$ be a graph   isomorphic to one of the graphs in  $\{W_4^-, W_4, 2K_1+P_4,  K_2+2K_2,  K_2+(K_1\cup K_3), W_5^<,  W_5^-, W_5,   K_7^<, K_7^-, K_7, K_1+(K_1\cup K_5) \}$.   If $G$ is $H$-free, then $h_d(G) \geq \lceil n/2\rceil$.
 \end{restatable}

Our proof of Theorem~\ref{t:main} in the case $H=K_2+2K_2$ relies on    Theorem~\ref{t:2K2}. Theorem~\ref{t:main} offers  a much simpler proof of the result obtained by Bosse~\cite{Bosse} on $W_5$-free graphs with independence number two and provides further support for both    Dominating Hadwiger's Conjecture and  Hadwiger's Conjecture.   Plummer, Stiebitz and Toft~\cite{PST03} and Kriesell~\cite{Kriesell} proved that  Hadwiger's Conjecture holds for all $H$-free graphs $G$ with $\alpha(G)\le 2$, where $H$ is any graph on at most five vertices with $\alpha(H)\le2$. However, their arguments in certain cases rely on the existence of a dominating connected matching; this technique  does not extend to constructing dominating clique minors. Theorem~\ref{t:main} immediately implies the following: \medskip

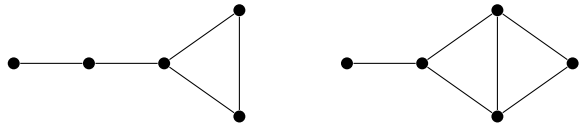
\begin{figure}[ht]
    \centering
    \begin{minipage}{0.2\textwidth}
        \centering
        \begin{tikzpicture}[scale=1, every node/.style={circle, fill=black, inner sep=1.6pt}]
			\node (v1) at (-1,0) {};
			\node (v2) at (0,.707) {};
			\node (v3) at (0,-.707) {};
			
			\draw (v1) -- (v2) -- (v3) -- (v1);
			\node (a) at (-2,0) {};
			\node (b) at (-3,0) {};
			\draw (b) -- (a) -- (v1);    \end{tikzpicture}

    \end{minipage}
    \hskip 1cm
    \begin{minipage}{0.2\textwidth}
        \centering
       \begin{tikzpicture}[scale=1, every node/.style={circle, fill=black, inner sep=1.6pt}]
			\node (v1) at (-2,0) {};
			\node (v2) at (-1,0) {};
			\node (v3) at (0,.707) {};
			\node (v4) at (0,-.707) {};
			\node (v5) at (1,0) {};
			
			\draw (v2) -- (v3) -- (v5) -- (v4) -- (v2);
			\draw (v3) -- (v4);
			\draw (v1) -- (v2);
		\end{tikzpicture}  
		  
    \end{minipage}
     \label{f:kh}
     \caption{ Hammer and Kite}
\end{figure}
 \begin{cor}\label{c:main}
 Let $G$ be a graph on $n$ vertices with $\alpha(G)\le 2$. Let $H\ne K_2\cup K_3$ be a graph with $|H|\le5$ and $\alpha(H)\le2$. If $G$ is $H$-free, then $h_d(G) \geq \chi(G)$.
 \end{cor}
 
 \begin{proof} One can easily verify  that there are seven graphs on four vertices with independence number two: $2K_2$, $K_1\cup K_3$, $P_4$, $K_1+(K_1\cup K_2)$, $C_4$, $K_4^-$ and $K_4$. Similarly, there are  $14$ graphs on five vertices with independence number two: $K_1\cup K_4$, $K_2\cup K_3$, $K_1+(K_1\cup K_3)$,  the Hammer, the Kite, $C_5$, $K_1+2K_2$, $K_1+P_4$, $W_4^<$, $W_4^-$, $W_4$, $K_5^<$, $K_5^-$ and $K_5$, where  the  Hammer and Kite graphs are given in Figure~\ref{f:kh},  and note that $W_5^<$ contains an induced copy of both. 
 \end{proof}
 It remains unknown whether   Conjecture~\ref{c:dDM} holds for $(K_2\cup K_3)$-free graphs. This paper is organized as follows. We first prove in Section~\ref{s:highmindeg} several lemmas on dominating clique minors in graphs with high minimum degree. 
   We then prove Theorem~\ref{t:omega} in Section~\ref{s:omega}. Finally, we use Theorem~\ref{t:omega} and the results in Section~\ref{s:highmindeg} to prove Theorem~\ref{t:main} in Section~\ref{s:main}.  \medskip
   
   We need to introduce more notation. A \dfn{clique} in a graph  $G$ is a set   of pairwise adjacent vertices. For a vertex $x \in V(G)$, we use $N(x)$ to denote the set of vertices in $G$ which are all adjacent to $x$. We define $N[x]$ to be  $N(x) \cup \{x\}$. The degree of $x$ is denoted by $d_G(x)$ or simply $d(x)$. If $A,B \subseteq V(G)$ are disjoint, we say that $A$ is \dfn{complete} to $B$ if  every vertex in $A$ is adjacent to all vertices in $B$, and $A$ is \dfn{anticomplete} to $B$ if no vertex in $A$ is adjacent to any vertex in $B$. If $A = \{a\}$, we simply say $a$ is complete to $B$ or $a$ is anticomplete to $B$. We use $e_G(A,B)$ to denote the number of edges between $A$ and $B$ in $G$. The subgraph of $G$ induced by $A$ is  denoted as $G[A]$. We denote by $B \less A$ the set $B - A$  and $G \less A$ the subgraph of $G$ induced on $V(G) \less A$.     If $A = \{ a\}$, we simply write $A \setminus a$ and $ G \setminus a$, respectively. An edge $e=xy$ in $G$ is \dfn{dominating} if every vertex in $V(G)\less\{x,y\}$ is adjacent to $x$ or $y$. Throughout, we use the notation ``$S :=$'' to indicate that $S$ is defined to be the expression on the right-hand side. For any positive integer $k$, we write $[k]:=\{1, \dots , k\}$.

\section{Graphs   with high minimum degree}\label{s:highmindeg}

 We begin this section  with a lemma concerning graphs without constraints on  their independence number. 
 
\begin{lem}\label{l:mindeg}
    Let $G$ be a graph on $n$ vertices. If $\delta(G) \geq n-3$, then $h_d(G) \geq \chi(G)$.
\end{lem}
\begin{proof}
    Suppose for a contradiction that $h_d(G)< \chi(G)$. We choose $G$  with $n$ minimum.  
    Since $\delta(G) \geq n-3$, we see that  $\alpha(G)\le3$ and $\Delta(\overline{G}) \leq 2$. Thus $\overline{G}$ is the disjoint union of cycles and paths. Let $S$ be a maximum stable set in $G$. Suppose $|S|=3$. Then $S$ is complete to $V(G)\less S$ because $\delta(G)\ge n-3$. Note that $\chi(G\less S)= \chi(G)-1$. By the minimality of $n$, $ G\less S $ has a dominating clique minor of order $\chi(G)-1$, say $(T_2, \ldots, T_{\chi(G)})$. Let   $s\in S$. Then $(G[\{s\}], T_2,   \ldots,  T_{\chi(G)})$ yields a dominating clique minor of order $\chi(G)$ in $G$, a contradiction. Thus $\alpha(G)=|S|=2$. We claim that $\overline{G}$ is $2$-regular. Suppose there exists a vertex  $x\in V(G)$ such that $d(x)\ge n-2$. Then $\chi(G[N(x)])=\chi(G)-1$. By the minimality of $n$, $h_d(G[N(x)])\ge\chi(G[N(x)])$ and thus $h_d(G)\ge h_d(G[N(x)])+1\ge\chi(G[N(x)])+1= \chi(G)$, a contradiction. This proves that $\overline{G}$ is $2$-regular, as claimed. Suppose next $\overline{G}$ has no odd cycles. Then $\overline{G}$ is bipartite, $n$ is even,   $\chi(G)= n/2$, and $h_d(G)\ge \omega(G)\ge n/2$, a contradiction. Thus $\overline{G}$ contains at least one odd cycle, say $C$ with vertices $v_1, \ldots,   v_{2k+1}$ in order, where $k\ge 1$ is an integer. Since $\alpha(G)=2$, we see that  $k\ge2$. Then $\chi(G)=(k+1)+\chi(G\less V(C))$. By the minimality of $n$,    $G\less V(C)$ has a dominating clique minor of order $\chi(G)-(k+1)$, say $(T_{\chi(G)-(k+2)}, \ldots, T_{\chi(G)})$.  Then $(G[\{v_1, v_3, v_{2k+1}\}], G[\{v_2\}],  G[\{v_4\}], \ldots, G[\{v_{2k}\}],   T_{\chi(G)-(k+2)}, \ldots, T_{\chi(G)})$ yields a dominating clique minor of order $\chi(G)$ in $G$, a contradiction. 
\end{proof}

\begin{lem}\label{l:highmindeg} Let $G$ be a graph  on $n$ vertices with $\alpha(G)\le2$. If $\delta(G) \ge n-6$, then $h_d(G)\ge \ceil{n/2}$.  
\end{lem}
\begin{proof} Let $G$ be given as in the statement. By Corollary~\ref{c:smalln}, we may assume that $n\ge 27$.   Let $x\in V(G)$ with $d(x)=\delta(G)$ and $A:=V(G)\less N[x]$. Then $d(x)\ge n-6$ and $A$ is a  clique of order $n-1-d(x)\le 5$. We next show that   some vertex  in $N(x)$ is complete to $A$ in $G$. Suppose no vertex in $N(x)$ is complete to $A$ in $G$. Then $e_{\overline{G}}(N(x), A)\ge |N(x)|\ge n-6\ge21$ since $n\ge27$, and so some vertex, say $z$, in $A$ is not adjacent to at least $\ceil{e_{\overline{G}}(N(x), A)/|A|}\ge\ceil{21/5}\ge5$ vertices in $N(x)$.  
Note that $zx\notin E(G)$.  Thus  $d(z)\le n-7$,  contrary to the assumption that $\delta(G)\ge n-6$. This proves that some vertex, say $y$,  in $N(x)$ is complete to $A$ in $G$. Then the edge $xy$ is dominating in $G$. Note that  $\delta(G\less\{x,y\})\ge \delta(G)-2\ge (n-2)-6$. By the induction hypothesis, $G\less\{x,y\}$ has a dominating clique minor of order $\ceil{(n-2)/2}$, say $(T_2, \ldots, T_{\ceil{n/2}})$. Then $(G[\{x,y\}], T_2, \ldots, T_{\ceil{n/2}})$ yields a dominating clique minor of order $\ceil{n/2}$ in $G$,  as desired. 
\end{proof}

Lemma~\ref{l:highmindeg} immediately implies the following:

\begin{cor}\label{c:con} Let $G$ be a graph  on $n$ vertices with $\alpha(G)\le2$. If $G$ is $(n-6)$-connected, then $h_d(G)\ge \ceil{n/2}$.  
\end{cor}
 
 We end this section with two lemmas that are used in the proof of Theorem~\ref{t:main}. 
 
 \begin{lem}\label{l:C5} Let $G$ be a graph on $n$ vertices with $\alpha(G)\le2$.
   If $\omega(G)<\ceil{n/2}$ and $G$ has no dominating edge, then every vertex in $G$ with degree less than $n-1$  lies on  an induced $C_5$ in $G$. 
   \end{lem}
 \begin{proof} Let $x\in V(G)$ with $d(x)\le n-2$.  Since $\alpha(G)\le2$ and $\omega(G)<\ceil{n/2}$, we see that  $\alpha(G[N(x)])=2$ and $V(G)\less N[x]$ is a clique. Let $y,z\in N(x)$ with $yz\notin E(G)$.  Since  $G$ has no dominating edge,  there exist $ y',z'\in V(G)\less N[x]$ such that $zz', yy'\notin E(G)$. Then $yz', zy'\in E(G)$ and so $y'\ne z'$. It follows that $G[\{x, y,z, y',z'\}]=C_5$, as desired.
 \end{proof}  
 
 \begin{lem}\label{l:C4} Let $G$ be a graph on $n$ vertices with $\alpha(G)\le2$.
   If $2\omega(G)\le\ceil{n/2}$, then every vertex in $G$ with degree less than $n-1$  lies on  an induced $C_4$ in $G$. 
   \end{lem}
 \begin{proof} Let $x\in V(G)$ with $d(x)\le n-2$.  Let $y\in V(G)\less N[x]$.   Since $\alpha(G)\le2$ and $2\omega(G)\le\ceil{n/2}$, we see that    both $V(G)\less N[x]$ and $V(G)\less N[y]$ are cliques of order at most $\omega(G)$. Then $|N(x)\cap N(y)|\ge n-(|V(G)\less N[x]|+|V(G)\less N[y]|)\ge n-2\omega(G)\ge\lfloor n/2\rfloor$. If $N(x)\cap N(y)$ is a clique, then $\omega(G)\ge |(N(x)\cap N(y))\cup\{x\}|\ge \ceil{n/2}$, contrary to the assumption that $2\omega(G)\le\ceil{n/2}$.   Thus there exist $z,w\in N(x)\cap N(y)$ such that $zw\notin E(G)$, and so $G[\{x, y,z, w\}]=C_4$, as desired.
 \end{proof}  
 
  \section{Proof of Theorem~\ref{t:omega}}\label{s:omega}

A \dfn{seagull} in a graph is an induced path on three vertices.  For a nonempty clique $K$ in a graph $G$, let $K^*$  be the set of vertices $v \in V (G) \less K$ such that   $v$ is neither complete nor anti-complete to $K$.   We define the \dfn{capacity of $K$}  to be 
$(|G|+|K^*|-|K|)/2$.  Theorem~\ref{t:seagulls} is a deep result of Chudnovsky and Seymour~\cite{seagull} on the existence of $\ell$ pairwise disjoint seagulls in graphs with independence number two.

\begin{thm}[Chudnovsky and Seymour~\cite{seagull}]\label{t:seagulls}  Let $G$ be a graph with $\alpha(G)\le2$ and let $\ell\ge0$  be an integer such that
 $G\ne W_5$  when $\ell=2$. Then $G$ has $\ell$ pairwise disjoint seagulls  if and only if
\begin{enumerate}[(i)]
\item  $|G|\ge 3\ell$,
\item $G$ is $\ell$-connected,
\item every clique of $G$ has capacity at least $\ell$ and
\item  $G$ admits an anti-matching of cardinality $\ell$, that is,  $\overline{G}$ admits a   matching of cardinality $\ell$.
\end{enumerate}
\end{thm}

Theorem~\ref{t:seagulls} plays a key role in proving Theorem~\ref{t:omega}, which we restate below for convenience. \Clique*

 \begin{proof} Suppose $h_d(G) < \lceil n/2\rceil$. We choose $G$ with $n$ minimum. Then $G$ is connected, $\alpha(G)=2$  and $\omega(G)\le h_d(G) < \lceil n/2\rceil$. Suppose $n$ is even.  Let  $x\in V(G)$ such that $\omega(G\less x)=\omega(G)$. This is possible because  $\alpha(G)=2$.  Then $2\omega(G\less x)=2\omega(G)\ge \ceil{n/2}+1=\ceil{(n-1)/2}+1$.   By the minimality of $n$,    $h_d(G)\ge h_d(G\less x) \ge \ceil{(n-1)/2}=\ceil{n/2}$, a contradiction. Thus $n$ is odd. 
Recall  that $\omega(G)< \ceil{n/2}$.  Let  $\omega(G):= \ceil{n/2}-\ell$ for some  integer $\ell\ge1$. Then $n=2\omega(G)+2\ell-1$. Since $2\omega(G)\ge \ceil{n/2}+1$, we see that     $n\ge 4\ell+1$. \medskip

Let $A$ be a maximum clique in $G$ and let $G' := G \backslash A$. Then $\omega(G') \leq \omega(G)$. 
Suppose    $G'$ is disconnected. 
 Then $V(G)\less A$ is the disjoint  union of two cliques, say, $B_1$ and $B_2$, such that  $B_1$ is anticomplete to $B_2$ in $G$. Since $\alpha(G)=2$, we see that  each vertex in $A$ is either complete to $B_1$ or complete to $B_2$. Define
\[A':=\{a\in   A\mid a \text{ is  complete to } B_1\} \text{ and } A'':= A\less A'.\]
Then $A'\cup B_1$ and $A''\cup B_2$ are disjoint cliques in $G$, and $V(G)=(A'\cup B_1)\cup (A''\cup B_2)$. Thus  $\omega(G)\ge \max\{| A'\cup B_1|, |A''\cup B_1|\}\ge \ceil{n/2}$, a contradiction. This proves that  $G'$ is connected. 
 If some vertex in $A$ is anticomplete to $V(G')$, then  $V(G')$ is a clique in $G$ and so  $\omega(G)\ge |G'|>\ceil{n/2}$, a contradiction. Thus  every vertex in $A$ is adjacent to at least one vertex in $G'$.  Then $G'$, together with $A$, yields a dominating clique minor of order $\omega(G)+1$ in $G$. It follows that $ \omega(G)+1\le h_d(G)<\ceil{n/2}$. Thus  $\omega(G)\le \ceil{n/2}-2$ and $|G'|\ge \ceil{n/2}+1$.  Suppose $G'= W_5$. Then $n=9$. Since the Ramsey number $R(3,4)=9$, we see that  $\omega(G)\ge4$, contrary to the fact that $\omega(G)\le \ceil{n/2}-2=3$.  This proves that $G'\ne W_5$. We next show that $G'$ contains $\ell  $ pairwise disjoint seagulls. \medskip
 
 Note  that $|G'| =n-\omega(G)=n-(\ceil{n/2}-\ell)=\lfloor n/2\rfloor+\ell\geq 3\ell$ 
 because $n\ge 4\ell+1$. Moreover, $G'$ admits an $\ell$ anti-matching, else let $M'$ be an anti-matching in $G'$ with $|M'|\le \ell-1$. Then 
\begin{align*}
\omega(G')&\ge |G'|-|V(M')|\\
	          &\ge |G'|-2(\ell-1)\\
	          &= (n-\omega(G))-2(\ell-1) \\
	          &=(2\omega(G)+2\ell-1)-\omega(G)-2(\ell-1)\\
	          &= \omega(G) +1,
\end{align*}	        
a contradiction.   Let  $K$ be a clique  in $G'$. Then  $|K| \leq \omega(G)$. We next show that   $K$   has capacity at least $\ell$ in $G'$. Recall that  $K^*$ denotes the set of vertices $v$ in $G'\less K$ such  that  $v$ is neither complete nor anticomplete to $K$. We claim that  $|K|-|K^*|\le \omega(G)-1$. This is trivially true   when $|K| \le  \omega(G)-1$. Suppose   $|K| = \omega(G)$. Then $\omega(G')=|K|$ and  $|G' \backslash K| = n-2\omega(G)=2\ell-1\ge1$. It suffices to show that  $|K^*|\ge1$. Since $G'$ is connected and $|G' \backslash K|\ge1$, there exists a vertex $u\in V(G')\less K$ such that $u$ is adjacent to some vertex in    $K$.  Then $u$ is not complete to $K$ in $G'$ because  $\omega(G')=|K|$. Thus $u$ 
is neither complete nor anticomplete to $K$ in $G'$, and so $|K^*|\ge1$.   This proves that  $|K|-|K^*|\le \omega(G)-1$, as claimed.  It follows that  $K$ has capacity $(|G'|+|K^*|-|K|)/2\ge \big((n-\omega(G))-(\omega(G)-1)\big)/2=(n-2\omega(G)+1)/2=\ell$ in $G'$. Finally, suppose $G'$ is not $\ell$-connected. Let $S$ be a minimum vertex cut in $G'$. Then $|S| \leq \ell-1$ and $G' \backslash S$ has exactly two components, say $G_1$ and $G_2$. Note that $V (G_1)$ and $V (G_2)$ are disjoint cliques; moreover, $V (G_1)$ is anticomplete to $V (G_2)$ in $G$. We   partition $A$ into $A'$ and $A''$ such that $V (G_1) \cup A'$
and $V (G_2) \cup  A''$ are disjoint cliques. But then 
\begin{align*}
\omega(G)&\geq  \max\{|V (G_1) \cup A'|, |V (G_2) \cup  A''|\} \\
&\geq \ceil{|G\less S|/2} \\
&\ge\ceil{\big(n-(\ell-1)\big)/2}\\
&=\ceil{\big((2\omega(G)+2\ell-1)-\ell+1\big)/2}\\
&>\omega(G),
\end{align*}
 a contradiction. This proves that $G'$ is $\ell$-connected. By Theorem~\ref{t:seagulls} applied to $G'$ and $\ell$, we see
that $G'$ has $\ell$ pairwise disjoint seagulls. Such $\ell$ seagulls, together with the clique $A$, yield a dominating clique minor of order $\omega(G)+\ell=\lceil n/2\rceil$ in $G$, a contradiction. \end{proof}

 
 \section{Proof of Theorem~\ref{t:main}}\label{s:main}
 
 In this section we prove Theorem~\ref{t:main} which we restate below for convenience. \GraphH*
 
 \begin{proof}     Suppose the statement is false. We choose $G$ with $n$ minimum. Then $G$ is connected  and $\omega(G)\le h_d(G) < \lceil n/2\rceil$.  We next prove several claims. By  Corollary~\ref{c:smalln}, 
 \medskip

 \setcounter{counter}{0}
\noindent {\bf Claim\refstepcounter{counter}\label{n}  \arabic{counter}.}
   $n\ge27$.\medskip

\noindent {\bf Claim\refstepcounter{counter}\label{nodd}  \arabic{counter}.}   
$n$ is odd. 
 
 \begin{proof}Suppose $n$ is even. Let  $x\in V(G)$. By the minimality of $n$,    $h_d(G)\ge h_d(G\less x) \ge \ceil{(n-1)/2}=\ceil{n/2}$, a contradiction. 
 \end{proof}

\noindent {\bf Claim\refstepcounter{counter}\label{clique}  \arabic{counter}.} $2\omega(G)\le\ceil{n/2}$ and 
  $7\le \omega(G) \leq \ceil{n/2} -7$.    Moreover, $\omega(G)  =\ceil{n/2} -7$ if and only if  $n=27$. 

\begin{proof}  By Theorem~\ref{t:omega}, $2\omega(G)\le\ceil{n/2}$. By Corollary~\ref{c:smallclique}, $\omega(G)\ge7$. It follows that 
\[7\le \omega(G) \leq \ceil{n/2} -\omega(G)\le \ceil{n/2}-7.\]
 Thus  $ \omega(G) = \ceil{n/2} -7=7$ if $n=27$. Next suppose $\omega(G)=\ceil{n/2} -7$. Since $2\omega(G)\le\ceil{n/2}$, we have $2\big(\ceil{n/2} -7\big)\le\ceil{n/2}$,  and so $n=27$.  
\end{proof}

\noindent {\bf Claim\refstepcounter{counter}\label{d-edge}  \arabic{counter}.}
  $G$ contains no dominating edge.
 
\begin{proof}
    Suppose there exists an edge $e=xy\in E(G)$ such that  $e$ is dominating. Let $D_1:=\{x,y\}$.   By the minimality of $n$, $G\less D_1$ contains a dominating clique minor of order $\ceil{n/2}-1$, say $(T_2, \ldots, T_{\ceil{n/2}})$. But then  $(G[D_1], T_2, \ldots, T_{\ceil{n/2}})$ yields a dominating clique minor of order $\ceil{n/2}$ in $G$, a contradiction.  
\end{proof}

\noindent {\bf Claim\refstepcounter{counter}\label{K}  \arabic{counter}.} 
For each non-empty clique $K$ in $G$, $K$ is not complete to $V(G)\less K$ in $G$, $G\less K$ is connected and  $\alpha(G\less K)=2$. 
\begin{proof}
 Let $K\ne\es$ be a clique in $G$. Suppose $K$ is  complete to $V(G)\less K$ in $G$. By the minimality of $n$, $G\less K$ contains a dominating clique minor of order $\ceil{(n-|K|)/2}$. This, together with $K$, yields a dominating  clique minor of order $\ceil{n/2}$ in $G$, a contradiction. 
Next, suppose  $G\less K$ is disconnected. 
 Then $V(G)\less K$ is the disjoint  union of two cliques, say, $A$ and $B$. Then each vertex in $K$ is either complete to $A$ or complete to $B$. Define
\[K':=\{u\in  K\mid u \text{ is   complete to } A\} \text{ and } K'':= K\less K'.\]
Then $K'\cup A$ and $K''\cup B$ are disjoint cliques in $G$. Thus $h_d(G)\ge\omega(G)\ge \ceil{n/2}$, a contradiction. This proves that $G\less K$ is connected. Since  $\omega(G)<\ceil{n/2}$ and $|G\less K|>\ceil{n/2}$, we see that   $\alpha(G\less K)=2$. 
\end{proof}

\noindent {\bf Claim\refstepcounter{counter}\label{nbr}  \arabic{counter}.} 
 For any $v\in V(G)$, $V(G) \less N[v]$  is a clique and $N(v)$ is not a clique. Furthermore, $\alpha(G[N(v)]) = 2$ and $G[N(v)]$ is connected.
\begin{proof}
Let $v\in V(G)$.  Since $\alpha(G)=2$, we see that   $V(G)\less N[v]$ is a clique. By Claim~\ref{K},   $N(v)$ is not a clique and  
 $\alpha(G[N(v)]) = 2$. Suppose  $G[N(v)]$ is disconnected. Then $N(v)$ is the disjoint union of two   cliques, say $A_1$ and $A_2$.  Thus   $ \omega(G)\ge\max\{|A_1|, |A_2|, |V(G)\less N[v]|\}\ge\ceil{(n-1)/3}$, contrary to Claim~\ref{clique}.  
\end{proof}

\noindent {\bf Claim\refstepcounter{counter}\label{degree}  \arabic{counter}.}
 $\max\{\ceil{n/2}+5, \ceil{(3n-5)/4}\}\le \delta(G)\le n-7$ and $\Delta(G)\le n-5$. 
 \begin{proof}
 By  Lemma~\ref{l:highmindeg}, $\delta(G)\le n-7$. Let  $v\in V(G)$ with $d(v)=\delta(G)$. By Claim~\ref{nbr}, $V(G) \less N[v]$  is a clique.  By Claim~\ref{clique}, $2|V(G) \less N[v]|\le 2\omega(G)\le \max\{\ceil{n/2}, 2(\ceil{n/2}-7)\}$. Thus 
 \[2(n-1)=2\left(d(v)+|V(G) \less N[v]|\right)\le  2d(v)+2\omega(G)\le  2d(v)+\max\{\ceil{n/2}, 2(\ceil{n/2}-7)\}.\] It follows that $\delta(G)\ge \max\{\ceil{n/2}+5, \ceil{(3n-5)/4}\}$. \medskip
 
 Next let  $x\in V(G)$ with $d(x)=\Delta(G)$.    By Claim~\ref{d-edge}, $d(x)\le n-2$.  Suppose $d(x)\ge n-4$.  Then   $|V(G)\less N[x]|\le3$. Since $\delta(G)\ge \ceil{(3n-5)/4}$ and $n\ge 27$, we see that there exists $y\in N(x)$ such that $y$ is complete to $V(G)\less N[x]$. Then  $xy$ is a dominating edge in $G$, contrary to Claim~\ref{d-edge}. 
 \end{proof}

\noindent {\bf Claim\refstepcounter{counter}\label{bowtie}  \arabic{counter}.} 
 $H \ne K_2+2K_2$.

\begin{proof}
 By Theorem~\ref{t:2K2}, $G$ contains an induced $2K_2$. Let    $X:=\{x_1, x_2, x_3, x_4\}\subseteq  V(G)$ such that $x_1x_2, x_3x_4\in E(G)$ and $G[X]=2K_2$. Let $Y:= V(G)\less X$. By Claim~\ref{degree}, $|N(x_i)|\ge (3n-5)/4$ for each $i\in[4]$. It follows that 
 \[|N(x_1)\cap N(x_2)|\ge |N(x_1)\cap Y|+|N(x_2)\cap Y| -|Y|\ge 2\big((3n-5)/4-1\big)-(n-4)=(n-1)/2.\]
  Similarly,  $|N(x_3)\cap N(x_4)|\ge (n-1)/2$. Let $Z:=N(x_1)\cap N(x_2)\cap N(x_3)\cap N(x_4)$. Then 
  \[|Z|\ge |N(x_1)\cap N(x_2)|+|N(x_3)\cap N(x_4)|-|Y|\ge (n-1)/2+(n-1)/2-(n-4)=3.\] Let $u, v\in Z$ such that $uv\in E(G)$.  Then $G[X\cup \{u, v\}]=K_2+2K_2$. Thus  $H\ne K_2+2K_2$.  
  \end{proof}

  \noindent {\bf Claim\refstepcounter{counter}\label{W4}  \arabic{counter}.} 
 $H \ne W_4$. 
 \begin{proof}
 Suppose   $H = W_4$.   By Lemma~\ref{l:C4},   $G$ contains an induced $C_4$, say with vertices $v_1, v_2, v_3, v_4$ in order.  Then no vertex in $V(G)\less V(C_4)$ is complete to $V(C_4)$ because $G$ is $W_4$-free. Since $\alpha(G)=2$, we see that each vertex in $V(G)\less V(C_4)$ is adjacent to  at least two consecutive vertices on $C_4$. For each $i\in[4]$, let 
 \[A_i:=\{v\in V(G)\less V(C_4)\mid N(v)\cap V(C_4)=V(C_4)\less\{v_i\} \}  \text{ and}\]
 \[ B_i:=\{v\in V(G)\less V(C_4)\mid vv_i, vv_{i+1}\notin E(G)\},\] where where all arithmetic on indices    is done modulo $4$.  Note that $A_1, A_2, A_3, A_4, B_1, B_2, B_3, B_4$ are pairwise disjoint cliques in $G$  and $\{A_1, A_2, A_3, A_4, B_1, B_2, B_3, B_4\}$ partitions $V(G)\less V(C_4)$.    Note that  $V_1:=\{v_1\}\cup B_2\cup A_3\cup B_3$, $V_2:=\{v_3\}\cup B_4\cup A_1\cup B_1$, $V_3:=\{v_1, v_4\}\cup   A_2$, and $V_4:=\{v_1, v_2\}\cup   A_4$  are   cliques in $G$ and $|V_1|+|V_2|+|V_3|+|V_4|=n+2$.   Thus 
 \[\omega(G)\ge \max\{|V_1|, |V_2|, |V_3|, |V_4|\}\ge   \ceil{(n+2)/4}\ge \ceil{(n+3)/4}\] since $n$ is odd, contrary to Claim~\ref{clique}.    
 \end{proof}
 
\noindent {\bf Claim\refstepcounter{counter}\label{W5}  \arabic{counter}.} 
 $H \ne W_5$. 
\begin{proof}
Suppose   $H = W_5$.   By Lemma~\ref{l:C5},   $G$ contains an induced $C_5$, say with vertices $v_1, v_2, v_3, v_4, v_5$ in order.  Then no vertex in $V(G)\less V(C_5)$ is complete to $V(C_5)$ because $G$ is $W_5$-free. Since $\alpha(G)=2$, we see that each vertex in $V(G)\less V(C_5)$ is adjacent to  at least three consecutive vertices on $C_5$. For each $i\in[5]$, let 
 \[A_i:=\{v\in V(G)\less V(C_5)\mid N(v)\cap V(C_5)=V(C_5)\less\{v_i\} \}  \text{ and}\]
 \[ B_i:=\{v\in V(G)\less V(C_5)\mid vv_i, vv_{i+1}\notin E(G)\},\] where where all arithmetic on indices    is done modulo $5$.  Note that $A_1, \ldots, A_5, B_1, \ldots, B_5$ are pairwise disjoint cliques in $G$  and $\{A_1, \ldots, A_5, B_1, \ldots, B_5\}$ partitions $V(G)\less V(C_5)$.  Since $\alpha(G) \leq 2$, we see that  $A_i$ is complete to $B_{i-1} \cup B_i$ and $B_i$ is complete to $A_i \cup A_{i+1} \cup B_{i-1} \cup B_{i+1}$. We next show that each vertex in $A_1$ is complete to $A_3$ or $A_5$.  Suppose there exist 	 
		 $a_1\in A_1$, $a_3\in A_3$ and $a_5\in A_5$ such that $a_1$ is anticomplete to $\{a_3, a_5\}$.  Then   $a_3a_5\in E(G)$ and $G[v_5,a_1,v_3,a_5, a_3, v_4]=W_5$, a contradiction. This proves that each vertex in $A_1$ is complete to $A_3$ or $A_5$. 
		Let 
\[A_1^3 := \{v \in A_1\mid  v \text{ is  complete to } A_3 \}  \text{ and }  A_1^5 := A_1 \setminus A_1^3.\] Then  $A_1^5$ is complete to $A_5$.  Note that 	 	  
$V_1 := A_1^5 \cup A_5 \cup B_5 \cup \{v_2,v_3\}$, 
 $V_2 := A_1^3 \cup A_3 \cup \{v_4,v_5\}$, 
 $V_3 := A_2 \cup B_1 \cup B_2\cup \{v_4,v_5\}$, and 
 $V_4 := A_4 \cup B_3 \cup B_4 \cup \{v_1, v_2\}$ 
are   cliques in $G$ and $|V_1|+|V_2|+|V_3|+|V_4|=n+3$.   Thus 
 \[\omega(G)\ge \max\{|V_1|, |V_2|, |V_3|, |V_4|\}\ge   \ceil{(n+3)/4},\] contrary to Claim~\ref{clique}.   
		\end{proof}

\noindent {\bf Claim\refstepcounter{counter}\label{H6*}  \arabic{counter}.}
  $H \ne W_5^<$. 
 \begin{proof}
   Suppose $H= W_5^<$. Then  $G$ is $W_5^<$-free. By Lemma~\ref{l:C5},   $G$ contains an induced $C_5$, say with vertices $v_1, v_2, v_3, v_4, v_5$ in order.  Let $X:=\{v_1, v_2, v_3, v_4, v_5\}$ and let \[J:=\{v\in V(G)\less X\mid v \text{ is complete to } X\}.\]  
   Since $G$ is $W_5^<$-free, no vertex in $V(G)\less (X\cup J)$ is adjacent to exactly three consecutive vertices on $C_5$. It follows that each vertex in $V(G)\less (X\cup J)$ is adjacent to  exactly  four vertices on $C_5$. For each $i\in[5]$, let $X_i:=\{v\in V(G)\less X\mid vv_i\notin E(G)\}$.   Then $\{X, X_1, \ldots, X_5, J\}$ partitions $V(G)$.  For each $i\in[5]$, since  $\delta(G) \leq n-7$,  we see that each $X_i$ is a clique of order at least four. 
 For the remainder of the proof of Claim~\ref{H6*}, all arithmetic on indices is done modulo $5$. For each $i,j\in[5]$ with $i\ne j$, we use $ G[X_i, X_j]$ to denoted the bipartite graph with bipartition $\{X_i, X_j\}$ such that  $ E(G[X_i, X_j])$ consists of  all edges in $G$ with one end in $X_i$ and the other in $X_j$. Let $k:=|X_1|$. We claim that \medskip
 
 \noindent ($*$)  For each $i \in [5]$,  $|X_i|=k$ and $G[X_i, X_{i+1}]$ is $(k-1)$-regular. Furthermore, $X_i$ is complete to $X_{i-2}\cup X_{i+2}$ and no vertex in $J$ is complete to $X_i$. 
 
Suppose there exists   $x_i\in X_i$ such that $x_i$ is complete to $X_{i + 1}$. Then $x_iv_{i + 1} $ is a dominating edge in $G$, contrary to Claim~\ref{d-edge}. Let $x_i\in X_i$ and $x_{i+1}\in X_{i+1}$ such that   $x_ix_{i +1} \notin E(G)$.  Then $G[\{x_i, v_{i+1},v_i, x_{i+1},v_{i-2}\}]=C_5$. Since $G$ is $W_5^<$-free, every   vertex  $v\in X_{i+1}$ with $v\ne  x_{i+1}$ is adjacent  to exactly four vertices in $G[\{x_i, v_{i+1},v_i, x_{i+1},v_{i-2}\}]  $; in particular,    $vx_i\in E(G)$. Thus $x_i$ is complete to $ X_{i+1}\less x_{i+1}$ and so each vertex in $X_i$ is adjacent to exactly $|X_{i+1}|-1$ vertices in $X_{i+1}$. By symmetry,  each vertex in $X_{i+1}$ is adjacent to exactly $|X_{i}|-1$ vertices in $X_i$. Thus  $|X_i| = k$  and $G[X_i, X_{i+1}]$ is $(k-1)$-regular.    Suppose next some vertex $x_i\in X_i$ is not complete to $X_{i-2}\cup X_{i+2}$.  By symmetry, we may assume that there exists $x_{i+2}\in X_{i+2}$ such that $x_ix_{i+2}\notin E(G)$.  Let $x_{i-1}\in X_{i-1}$ be the unique non-neighbor of $x_i$ in $X_{i-1}$.   Then $x_{i-1}x_{i+2}\in E(G)$ and $G[\{x_{i-1},x_{i+2},v_{i-1},x_i,v_{i+2}, v_i\}]=W_5^<$, a contradiction. Thus $X_i$ is complete to $X_{i-2}\cup X_{i+2}$.  Finally, suppose   there exists a vertex $u \in J$ such that $u$ is  complete to $X_i$. Then   $uv_i$  dominates $G$, a contradiction. This proves ($*$). \medskip 

We next claim that \medskip
			 
 \noindent ($**$)  For each $i \in [5]$,  every  edge in $  G[X_i]$ dominates $V(G)\less X$.  Moreover,  for each vertex  $x_i\in X_i$, let  $x_{i-1}\in X_{i-1}$ and $x_{i+1}\in X_{i+1}$ such that $x_ix_{i-1}, x_ix_{i+1}\notin E(G)$. Then the edge $x_{i-1}x_{i+1}$ dominates $V(G)\less (X\cup\{x_i\})$.  \medskip
			 
Let $e=x_ix_i'\in E(G[X_i])$. By ($*$),   $e$ dominates $ X_1, \ldots, X_5$.  Suppose there exists a vertex $u \in J$ such that $u$ is anticomplete to $\{x_i, x_i'\}$. By ($*$), let $x_{i+2}\in X_{i+2}$ such that $ux_{i+2}\notin E(G)$.  Then  $ G[\{v_i,u,v_{i+2},x_i,x_{i+2}, x_i'\}]=W_5^<$, a contradiction. Next  let  $x_i\in X_i$,   $x_{i-1}\in X_{i-1}$ and $x_{i+1}\in X_{i+1}$ such that $x_ix_{i-1}, x_ix_{i+1}\notin E(G)$.  Then $x_{i-1}x_{i+1}\in E(G)$. Note that $x_{i-1}$ is complete to  $X_{i+1} \cup X_{i+2} \cup \left(X_i \setminus \{x_i\}\right)$ and $x_{i+1}$ is complete to  $X_{i-1} \cup X_{i-2}$. It remains to show that the edge $x_{i-1}x_{i+1}$ dominates $J$. Suppose there exists a vertex $u \in J$ such that $u$ is anticomplete to $\{x_{i-1}, x_{i+1}\}$.  Then $ux_i\in E(G)$ and $ G[\{v_{i-1},x_i,v_{i+1},x_{i-1},x_{i+1}, u\}]=W_5^<$, a contradiction. This proves ($**$).\medskip	
			 
For each $i\in [4]$, let $X_i:=\{x^i_1, \ldots, x^i_k\}$ such that $x_\ell^2$ is anticomplete to $\{x_\ell^1,x_\ell^3\}$ for each $\ell\in[k]$. This is possible because $G[X_2, X_1]$ and $G[X_2, X_3]$ are $(k-1)$-regular bipartite graphs by ($*$). By ($**$),  $ x^1_\ell x^3_\ell$  dominates $V(G)\less(X\cup x_\ell^2)$.  Let $Y:=(X\less v_5)\cup X_1\cup X_2\cup X_3\cup X_4$. Then $|Y|=4k+4$ and $\{Y, X_5\cup\{v_5\}, J\}$ partitions $V(G)$ and $v_1$ is complete to $X_5\cup\{v_5\}\cup J$. By the minimality of $n$, $G\less Y$ has a dominating clique minor of order $\ceil{n/2}-(2k+2)$, say, $(T_{2k+3}, \ldots, T_{\ceil{n/2}})$. Recall that $k=|X_1|\ge4$. Let $k:=2p+r$, where $p\ge2$ is an integer and $r\in\{0,1\}$. Let $Y_1 := \{v_2, v_3, v_4\}$ when $k$ is even,  and let $Y_1 := \{v_2, v_3, x^4_k\}$ and $Y_2 := \{v_4, x^2_k\}$ when $k$ is odd.  For  each $j\in[p]$,   let  $Y_{1+r+j}:=\{x^2_{2j-1}, x^2_{2j}\}$ and $Y_{p+1+r+j}:=\{x^4_{2j-1}, x^4_{2j}\}$.   For each $\ell\in [k]$, let  $Y_{k+1+\ell}:=\{x^1_\ell, x^3_\ell\}$. Finally, let $Y_{2k+2} := \{v_1\}$. Then $Y_1, \ldots, Y_{2k+2}$ are pairwise disjoint and $\{Y_1, \ldots, Y_{2k+2}\}$ partitions $Y$. By ($*$) and ($**$),   $(G[Y_1], \ldots, G[Y_{2k+2}])$ yields a dominating  $K_{2k+2}$ minor in $G$, and each vertex in  $V(G)\less Y$ is adjacent to at least one vertex in $Y_i$ for each $i\in [2k+2]$.  Thus $(G[Y_1], \ldots, G[Y_{2k+2}], T_{2k+3}, \ldots, T_{\ceil{n/2}})$ yields a dominating  clique minor of order $\ceil{n/2}$ in $G$, a contradiction. 	Thus $H \ne W_5^<$. 	\end{proof}

Throughout  the remainder of the proof, let $x, y$ be two non-adjacent vertices in $G$.  Since $\alpha(G)=2$, we see that every vertex in $V(G)\less \{x,y\}$ is adjacent to either $x$ or $y$. Let $A: =N(x)\backslash N(y) $, $B: =N(x)\cap N(y)$ and $C: =N(y)\backslash N(x) $.  By Claim~\ref{nbr},  both $A\cup\{x\}$ and $C\cup\{y\}$ are cliques in $G$.  \medskip

\noindent {\bf Claim\refstepcounter{counter}\label{ABC}  \arabic{counter}.}  
  $|B|\ge n-2\omega(G)\ge \max\{\lfloor n/2\rfloor, 14\}$ and $B$ is not a clique.   
\begin{proof}
   We may assume that $d(x)\le d(y)$. By Claim~\ref{degree}, $|C|\ge   |A|\ge3$.  By Claim~\ref{clique},  \[ |A \cup \{ x \}|+|C \cup \{ y \}|  \le 2\omega(G)\le \ceil{n/2}.\]
     Thus  
   \[|B|=n-(|A \cup \{ x \}|+|C \cup \{ y \}|) \ge n-2\omega(G)\ge n- \ceil{n/2}=\lfloor n/2\rfloor \ge13,\]
    since $n\ge27$. Suppose $|B|=13$. Then $n=27$ and $ |A \cup \{ x \}|+|C \cup \{ y \}|  = 2\omega(G)= \ceil{n/2}$.  Thus  $|A|=|C|=6$ and $d(x)=d(y)=19$. By the arbitrary choice of $x$ and $y$, $G$ is $19$-regular on $27$ vertices, which is impossible.  This proves that $|B|\ge14$. 
    Since $|B|\ge \lfloor n/2\rfloor$ and $2\omega(G)\le \ceil{n/2}$, we see that $B$ is not a clique in $G$.   
  \end{proof}

\noindent {\bf Claim\refstepcounter{counter}\label{2K1P4}  \arabic{counter}.}
  $H\ne 2K_1+P_4$. 
 \begin{proof} Suppose $H=2K_1+P_4$. Then $G[B]$ is $P_4$-free, else $G[B\cup\{x,y\}]$ is not $H$-free.  Thus $G[B]$ is perfect by the Strong Perfect Graph Theorem~\cite{SPGT} and so $\omega(G[B])=\chi(G[B])\ge \ceil{|B|/2}$. By Claim~\ref{ABC}, $|B|\ge n-2\omega(G)$. Thus $\omega(G[B])\ge \ceil{(n-2\omega(G)/2}=\ceil{n/2}-\omega(G)$. Then $\omega(G)\ge \omega(G[B])+1\ge\ceil{n/2}-\omega(G)+1$, which contradicts   Claim~\ref{clique}. 
 \end{proof}

\begin{figure}[ht]
    \centering
    \begin{minipage}{0.2\textwidth}
        \centering
       
\begin{tikzpicture}[scale=1.2, every node/.style={circle, fill=black, inner sep=2pt}]
    \node (v1) at (0,1) {};
    \node (v2) at (1,1) {};
    \node (v3) at (2,1) {};
    \node (v4) at (0,0) {};
    \node (v5) at (1,0) {};
    \node (v6) at (2,0) {};
    \node (v7) at (0.5,0.5) {};
    \node (v8) at (1.5,0.5) {};

    \foreach \i/\j in {v1/v2,v2/v3,  v4/v5, v5/v6,  v1/v4,   v3/v6, v1/v7, v2/v7,v4/v7, v5/v7, v2/v8,v3/v8, v5/v8, v6/v8}
        \draw[thick] (\i) -- (\j);

\draw[thick, bend left=30] (v1) to (v3);
\draw[thick, bend right=30] (v4) to (v6);
    
\end{tikzpicture}

    \end{minipage}
    \hskip 1cm
    \begin{minipage}{0.2\textwidth}
        \centering
       
\begin{tikzpicture}[scale=1.2, every node/.style={circle, fill=black, inner sep=2pt}]
    \node (v1) at (0,1) {};
    \node (v2) at (1,1) {};
    \node (v3) at (2,1) {};
    \node (v4) at (0,0) {};
    \node (v5) at (1,0) {};
    \node (v6) at (2,0) {};
    \node (v7) at (0.5,0.5) {};
    \node (v8) at (1.5,0.5) {};

    \foreach \i/\j in {v1/v2,v2/v3,  v4/v5, v5/v6,  v1/v4,   v3/v6, v1/v7, v2/v7,v4/v7, v5/v7, v2/v8,v3/v8, v5/v8, v6/v8, v3/v4}
        \draw[thick] (\i) -- (\j);

\draw[thick, bend left=30] (v1) to (v3);
\draw[thick, bend right=30] (v4) to (v6);

\end{tikzpicture}

    \end{minipage}
     \hskip 1cm
    \begin{minipage}{0.2\textwidth}
        \centering
            
\begin{tikzpicture}[scale=1.2, every node/.style={circle, fill=black, inner sep=2pt}]
    \node (v1) at (0,1) {};
    \node (v2) at (1,1) {};
    \node (v3) at (2,1) {};
    \node (v4) at (0,0) {};
    \node (v5) at (1,0) {};
    \node (v6) at (2,0) {};
    \node (v7) at (0.5,0.5) {};
    \node (v8) at (1.5,0.5) {};

    \foreach \i/\j in {v1/v2,v2/v3,  v4/v5, v5/v6,  v1/v4,   v3/v6, v1/v7, v2/v7,v4/v7, v5/v7, v2/v8,v3/v8, v5/v8, v6/v8, v3/v4, v1/v6}
        \draw[thick] (\i) -- (\j);

\draw[thick, bend left=30] (v1) to (v3);
\draw[thick, bend right=30] (v4) to (v6);
 
    \end{tikzpicture}
    
    \end{minipage}
     \caption{Three graphs on eight vertices}
    \label{f:H8}
\end{figure}
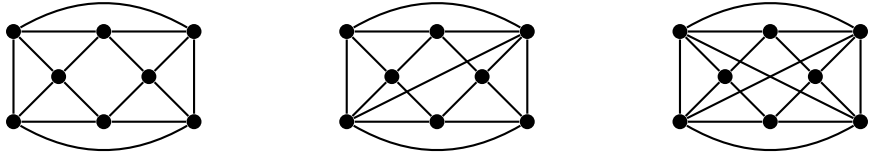

\noindent {\bf Claim\refstepcounter{counter}\label{W4-}  \arabic{counter}.} 
$G$ contains one of the graphs in Figure~\ref{f:H8} as an induced subgraph. Moreover, $H\ne W_5^-$ and $H\ne W_4^-$.
 \begin{proof}
    By Claim~\ref{ABC}, $|B|\ge \max\{\lfloor n/2\rfloor, 14\}$ and $B$ is not a clique in  $G$. Let $t\ge1$ be the largest integer such that there exist pairwise distinct vertices $z_1, \ldots, z_t, w_1, \ldots, w_t\in B$ such that $z_iw_i\notin E(G)$. Since all the edges $xz_1, xw_1, yz_1, yw_1$ are not dominating, there exist $z', w'\in A$ and $z'', w''\in C$ such that $z_1$ is anticomplete to $\{z', z''\}$ and $w_1$ is anticomplete to $\{w', w''\}$. Then $z'z'', w'w''\in E(G)$, $w_1$ is  complete to $\{z', z''\}$ and $z_1$ is complete to $\{w', w''\}$. It follows that $z'\ne w'$ and $z''\ne w''$. Then $G[\{x, y, z_1, w_1, z', z'', w', w''\}]$ is isomorphic to one of the graphs in  Figure~\ref{f:H8} depending on whether   $w'$ is adjacent to $z''$ and $z'$ is adjacent to $w''$ in $G$. Thus $H\ne W_5^-$ because $G[\{y, z_1,w',z', w_1, x\}]=W_5^-$. 
    Suppose $H= W_4^-$. We choose $x$ with $d(x)=\delta(G)$. Note that $G[\{x, y, z_i,w_i\}]=C_4$ for all $i\in [t]$. Then each vertex in $A\cup C$ is adjacent to exactly one vertex in  $\{z_i,w_i\}$, and each vertex in $B\less\{z_i,w_i\}$ is complete to $\{z_i, w_i\}$.  Let $Z:=\{z_1, \ldots, z_t\}$ and $W:=\{w_1, \ldots, w_t\}$. By symmetry, we may assume that $z'$ is anticomplete to $Z$. By the choice of $t$, $B\less (Z\cup W)$ is a clique in $G$.  Thus   $(B\less W)\cup\{x\}$ is a clique in $G$ and so $\omega(G)\ge |(B\less W)\cup\{x\}|=|B|-t+1$. By the choice of $x$, $d(z')\ge d(x)$. Note that $z'$ is anticomplete to $Z\cup \{y\}$. It follows that $t\le |C|\le \omega(G)-1$. But then 
    \[\omega(G) \ge |B|-t+1\ge(n-2-|A|-|C|)-|C|+1\ge n-1-3|C|\ge n-1-3(\omega(G)-1),\] which implies that $2\omega(G)\ge \ceil{n/2}+1$, contrary to Claim~\ref{clique}.    
 \end{proof}

\noindent {\bf Claim\refstepcounter{counter}\label{K1K5}  \arabic{counter}.} 
$H \ne  K_2 + (K_1\cup K_3)$. 

\begin{proof}  Suppose $H = K_2 + (K_1\cup K_3)$. We choose $x$ with $d_G(x)=\delta(G)$. Then  $|C|\ge5$. 
  For $z, w\in C$ with $z\ne w$, if there exist $u, v \in B\cap N(z)\cap N(w)$ such that $uv\in E(G)$,  then   $G[\{u, v,   x, y, z,w\}  ]=H$, a contradiction.  Thus $|B\cap N(z)\cap N(w)|\le2$ because  $\alpha(G)=2$.  It follows that  $e_G(\{z, w\}, B)\le |B|+2$.   
 Let \[B_1:= \{b \in B\mid b \text{ is anticomplete to } C\} \text { and } B_2:  = B\less B_1.\]
 Then $B_1$ is a clique. By Claim~\ref{nbr} applied to $G[N(y)]$,  $|B_2|\ge1$. Note that every vertex $b \in B_2$ is adjacent to  at least $|C|-2$  vertices  in $C$, else $G[\{b, y\} \cup C]$ is not  $H$-free.    Recall that for each pair of vertices $z,w$ in  $C$, $e_G(\{z, w\}, B_2)\le |B_2|+2$. Then some vertex in $B_2$ is adjacent to at most $(|B_2|+2)/2$ vertices in $C$ and so
 \[(|C|-2)|B_2|\le e_G(B_2, C)\le    \lfloor{|C|/2}\rfloor(|B_2|+2)+(\ceil{|C|/2}-\lfloor{|C|/2}\rfloor)(|B_2|+2)/2.\]
  It follows that $|B_2|\le  10$ because  $|C|\ge 5$. Then $|B_1|\ge 4$ by Claim~\ref{ABC}. Note that  $B_1$ and $C$ are two disjoint cliques and $B_1$ is anticomplete to $C$. We  now partition $A$ into $A'$ and $A''$ such that $A'\cup B_1$ and $A''\cup C$ are two disjoint cliques in $G$. But then   
  \[\omega(G)\ge \max\{|A'\cup B_1|, |A''\cup C|\}\ge \ceil{(n-|B_2|-2)/2}\ge \ceil{n/2}-6,\] contrary to Claim~\ref{clique}.  
 \end{proof}

\noindent {\bf Claim\refstepcounter{counter}\label{K1K5}  \arabic{counter}.} 
$H \ne K_1 + (K_1\cup K_5)$. 

\begin{proof}  
Suppose $H =K_1 + (K_1\cup K_5)$. We choose  $x$ with $d_G(x)=\delta(G)$.  Then $\delta(G)=|A|+|B|$ and  $|C|\ge5$.  
 Let $b\in B$.  If  $b$ is   not adjacent to at least five vertices in $B\cup C$, then $G[N[y]]$ is not  $H$-free. Thus $b$ is not adjacent to at most four vertices in   $B\cup C$.   Moreover, if 
  $b$ is adjacent to  four vertices in $C$, say $c_1, c_2, c_3, c_4$, then $G[\{b, x, y, c_1, c_2, c_3, c_4 \}]=H$, a contradiction. Thus $b$ is adjacent to at most three vertices in $C$.   Thus   $  |C|\le 7$  and 
  \[3(\delta(G)-|A|) =3|B|\ge e_G(B,   C)\ge |C|(\delta(G)-|C|-|A|).\] 
It follows that $(|C|-3)\delta(G)\le  (|C|-3)|A|+|C|^2$, and so 
\[ |A|+|B|=\delta(G) \le |A|+|C|+ \left\lfloor 3|C|/(|C|-3)\right\rfloor.\]
 Since  $5\le |C|\le 7$, we see that $|B|\le |C|+ \left\lfloor 3|C|/(|C|-3)\right\rfloor\le 12$, contrary to Claim~\ref{ABC}. 
 \end{proof}

 It remains to consider the cases $H\in\{K_7^<, K_7^-, K_7\}$. By Claim~\ref{clique}, $H\ne K_7$. By Claim~\ref{ABC}, $|B|\ge14$.   Since $R(3,5) = 14$, we see that  $G[B\cup\{x, y\}]$ contains an induced copy of $K_7^-$. Thus $H\ne K_7^-$.  Suppose $H=K_7^<$. We choose $x$ with $d(x)=\delta(G)$. Subject to the choice of $x$,  we choose $y$ with $d(y)$ minimum. 
 Let $z\in  C$. Suppose    $|N(z)\cap B|\ge9$. Then $G[N(z)\cap B]$ contains a $K_4$ subgraph since $R(3,4)=9$. Let $D$ be a clique of order four in $G[N(z)\cap B]$. Then $G[D\cup\{x, y, z\}]=K_7^<$, a contradiction.  Thus $|N(z)\cap B|\le8$ and so by the choice of $y$, 
 \[ |B|+|C|=d(y)\le d(z)\le 8+|C|+|A|.\]
   It follows that $|B|\le 8+|A|$.  Suppose $|B|\le |A|+6$. Then 
   \[\ceil{n/2}+5\le \delta(G)=|A|+|B|\le 6+2|A|\le 6+2(\omega(G)-1)\le 4+\ceil{n/2},\]
    which is impossible. Thus $7+|A|\le |B|\le 8+|A|$. Suppose next $|A|\le |C|-1$. Then $\omega(G)\ge |C|+1\ge |A|+2$ and so  
    \[\ceil{n/2}+5\le \delta(G)=|A|+|B|\le 8+2|A| \le 8+2(\omega(G)-2)\le 4+\ceil{n/2},\]
     which again is impossible. Thus $|A|=|C|$ and $d(y)=d(x)=\delta(G)$. Finally, suppose $|C|\le \omega(G)-2$. Then $\ceil{n/2}+5\le \delta(G)\le d(y)=|B|+|C|\le 8+2|C|\le 8+2(\omega(G)-2)\le 4+\ceil{n/2}$, a contradiction. Thus $|A|=|C|= \omega(G)-1$ and so $A\cup\{x\}$ is a maximum clique in $G$. Thus $z$ is not complete to $A$ and so 
   \[7+2|C|\le |B|+|C|=d(y)\le d(z)\le 8+|C|+(|A|-1)=7+2|C|.\]
   Then $|B|= 7+|C|$ and  $d(z)=d(y)=\delta(G)$.  Thus 
   \[n=|A\cup\{x\}|+|B|+ |C\cup\{y\}|=9+3|C|=6+3\omega(G)\le   6+1.5\times  \ceil{n/2},\] which implies that $n=27$ and $\omega(G)=7$. It follows that $|B|=7+|C|=13$, contrary to Claim~\ref{ABC}.
    \medskip
 
 This completes the proof of Theorem~\ref{t:main}.
 \end{proof}
 
 \noindent{\bf Acknowledgements.} Zi-Xia Song   thanks the Banff International Research Station and the organizers of  the ``New Perspectives in Colouring and Structure'' workshop, held  in Banff from September 29 - October 4, 2024, for their  kind  invitation and hospitality.

\end{document}